\documentclass[journal]{IEEEtran}

\usepackage{amsmath} 
\usepackage{amssymb}  
\usepackage{mathtools}
\usepackage{amsthm}
\usepackage{graphicx}
\usepackage{subcaption}
\usepackage{footnote}
\usepackage[font=small,labelfont=bf]{caption}
\usepackage{cite}

\newtheorem{theorem}{Theorem}
\newtheorem{corollary}{Corollary}
\newtheorem{lemma}{Lemma}

\theoremstyle{definition}
\newtheorem{remark}{Remark}

\newtheorem{definition}{Definition}

\newtheoremstyle{break}{3pt}{3pt}{}{}{\bfseries}{.}{\newline}{}
\theoremstyle{break}

\usepackage{tikz}
\usetikzlibrary{shapes,arrows,positioning,calc}

\tikzset{
block/.style = {draw, fill=white, rectangle, minimum height=2.5em, minimum width=3em},
tmp/.style  = {coordinate},
sum/.style= {draw, fill=white, circle, node distance=1cm},
input/.style = {coordinate},
output/.style= {coordinate},
pinstyle/.style = {pin edge={to-,thin,black}}
}

\title{\LARGE \bf
Reference Governors and Maximal Output Admissible Sets for Linear Periodic Systems
}

\author{Hamid~R.~Ossareh,~\IEEEmembership{Member,~IEEE}
\thanks{Hamid R. Ossareh is an assistant professor in the department of Electrical and Biomedical Engineering,   University of Vermont. E-mail:
        {\tt\small hamid.ossareh@uvm.edu}}}

\begin{document}

\maketitle

\begin{abstract}
In this paper, we consider the problem of constraint management in Linear Periodic (LP) systems using Reference Governors (RG). First, we present the periodic-invariant maximal output admissible sets for LP systems. We extend the earlier results in the literature to Lyapunov stable LP systems with output constraints, which arise in RG applications. We show that, while the invariant sets for these systems may not be finitely determined, a finitely-determined inner approximation can be obtained by constraint tightening. We further show that these sets are related via simple transformations, implying that it suffices to compute only one of them for real-time applications. This greatly reduces the memory burden of RG, at the expense of an increase in processing requirements. We present a thorough analysis of this trade-off. In the second part of this paper, we discuss two RG formulations and discuss their feasibility criteria and algorithms for their computation. Numerical simulations demonstrate the efficacy of the approach.

%
%
\end{abstract}


\IEEEpeerreviewmaketitle

\section{Introduction}

Linear periodic systems arise in many important practical applications such as aerospace \cite{Psiaki2001, Johnson2012, Wisniewski1999}, wind turbines \cite{Dugundji1983}, xerography \cite{Ching2010,Eun2013}, economic systems \cite{DeLima2001}, and multi-rate systems \cite{Vaidyanathan1993,Gondhalekar2009}. The study of periodic systems dates back to the 1800s, when Faraday studied periodic differential equations arising in physics \cite{Faraday1831}. The most rigorous theory of linear periodic systems is due to Floquet in 1883, who described stability conditions and invariant representations of periodic systems \cite{Floquet1883}.

In the control literature, several approaches for control of periodic systems have been developed. Controllability and eigenvalue assignment are addressed in \cite{Brunovsky1969, Kabamba1986}. Methods for stabilization are provided in \cite{Bittanti1985,Colaneri1998,DeSouza2000,Bittanti2001}. Reference \cite{Bittanti2000} contains a detailed summary of invariant representations of periodic systems. For a complete modern treatment of periodic systems and their control methods, please see \cite{BittantiSergio2009}. Recently, with the availability of high speed computers, many optimization-based and predictive control methods have been developed for periodic systems  \cite{Gondhalekar2009, Bohm2009, Freuer2010, Nguyen2014a,Nguyen2014}.

In the literature of linear time-invariant (LTI) systems, new predictive methods have emerged since the 1990s to specifically address the issue of constraint management. One such  method is the so-called Reference Governor (RG) \cite{Gilbert1991a, Gilbert1999,Kolmanovsky2014,Kalabic2015}, which is motivated by the challenge in designing controllers that achieve both tracking and constraint management. Model predictive control is applicable, but it is either difficult to calibrate by non-experts or is computationally expensive for real-time applications. As a result, many control designers decouple the problems of tracking and constraint management and approach them in a modular manner. This decoupling typically leads to a loss of performance. Reference governor theory aims to reduce this loss of performance by addressing constraint management in a systematic manner. Due to these benefits of RG, it is of interest to develop a rigorous RG theory for linear periodic systems.

A block diagram of a system with RG is shown in Fig. \ref{fig:RGblock}, where Closed-Loop Plant refers to the plant together with a pre-designed linear tracking controller, the output $y(t)$ refers to the constrained output (not the tracking output), $r(t)$ refers to the reference signal(s) to be tracked, $v(t)$ is the governed reference, and $x$ is the measured state (or an estimate obtained from an observer). The idea behind RG is to first characterize the set of all initial states and inputs that prevent constraint violation for all times. This set is referred to as the maximal output admissible set (MAS) and is computed offline. The reference governor then solves a computationally inexpensive linear program at every timestep to select an input such that the system state belongs to MAS and, therefore, satisfies the constraints for all times. 

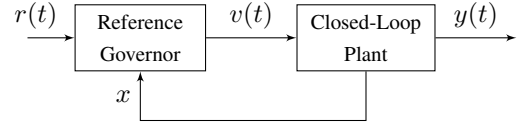
\begin{figure}
\centering
\begin{tikzpicture}[auto, node distance=1.5cm,>=latex']
    \node [input, name=rinput] (rinput) {};
    \node [block, right of=rinput,text width=1.5cm,align=center] (controller) {{\footnotesize Reference Governor}};
    \node [block, right of=controller,node distance=3cm,text width=1.6cm,align=center] (system)
{{\footnotesize Closed-Loop Plant}};
    \node [output, right of=system, node distance=2cm] (output) {};
    \node [tmp, below of=controller,node distance=1.1cm] (tmp1){$s$};
    \draw [->] (rinput) -- node{\hspace{-0.4cm}$r(t)$} (controller);
    \draw [->] (controller) -- node [name=v]{$v(t)$}(system);
    \draw [->] (system) -- node [name=y] {$y(t)$}(output);
    \draw [->] (system) |- (tmp1)-| node[pos=0.75] {$x$} (controller);
    \end{tikzpicture}
\caption{Reference governor block diagram} \label{fig:RGblock}
\end{figure}

In the literature for linear periodic systems, MAS has been previously characterized for asymptotically stable systems. Specifically, it is shown in  \cite{Freuer2010, Nguyen2014a,Nguyen2014,Blanchini1993} that, in contrast to LTI systems, which only require one set to describe MAS, periodic systems require $N$ sets, where $N$ is the system period. However, the relationship between these sets remains unexplored. Furthermore, the case of systems that are Lyapunov stable but not asymptotically stable, which arises in RG applications, has not been addressed. \textit{The goal of this paper is to develop an RG theory for linear periodic systems and address the above limitations of the current literature.} The original contributions of this paper are as follows:
\begin{itemize}
\item Development of MAS for Lyapunov stable periodic systems with output constraints;
\item Development of an inner approximation of MAS for Lyapunov stable periodic systems that is finitely determined (i.e., can be computed in finite time);
\item Investigation of geometric and algebraic properties of MAS for periodic systems -- it is shown that the $N$ sets that describe MAS are related via simple transformations, with the implication that only one of them needs to be stored in memory for real-time applications. This greatly reduces the memory requirements of  RG, at the expense of an increase in processing requirements;
\item Analysis of the above trade-off between memory and processing requirements (i.e., ``space-time'' trade-off, as commonly referred to in the computer science community);
\item Formulation of two reference governor designs for periodic systems, and investigation of their feasibility criteria, as well as algorithms for their computation. Numerical simulations are used to compare the two designs.
\end{itemize}

The paper is organized as follows. Section II describes the system equations, discusses stability and invariant representations of periodic systems, and states the standing assumptions in the paper. Section III presents the results on the maximal output admissible sets. Section IV presents the results on the reference governor formulations. Section V concludes with a brief summary and topics for future work.

Notations throughout this paper are standard. $\mathbb{R}$ and $\mathbb{Z}^+$ denote the set of real numbers and non-negative integers, respectively. All matrices are assumed to have real entries. The identity matrix in $\mathbb{R}^{n\times n}$ is denoted by $I_n$. The $i$-th standard basis vector in $\mathbb{R}^n$ (i.e., vector of all 0's except 1 in the $i$-th place), is denoted by $e_i$. The vector of all 1's (of appropriate dimensionality) is denoted by $\textbf{1}$, and the matrix, or vector, of all 0's (of appropriate dimensionality) is denoted by $\textbf{0}$. Superscript $^\top$ denotes matrix transpose.  The Euclidean spaces are equipped with the standard 2-norm: $\|x\|^2 = x^T x$. Open and closed sets are defined by the topology induced by the 2-norm. The interior of a set $A$ is denoted by int$(A)$.
All vector inequalities of the form $Ax\leq b$ are to be interpreted component-wise.

\section{Preliminaries}

Consider the discrete-time linear $N$-periodic system
\begin{equation}\label{eq:sys}
\begin{aligned}
x(t+1) &= A(t) x(t)\\
y(t) &= C(t) x(t)
\end{aligned}
\end{equation}
where $t \in \mathbb{Z}^+$ is the discrete time variable, $x(t)\in \mathbb{R}^n$ is the state, $y(t) \in \mathbb{R}^p$ is the constrained output, and matrices $A(t) \in \mathbb{R}^{n \times n}$ and $C(t) \in \mathbb{R}^{p \times n}$ are periodic with period $N$, i.e., $A(t+N) = A(t)$ and $C(t+N) = C(t)$. We assume the system starts from the initial condition $x(t_0)=x_0$. Systems with inputs, which arise in RG applications, can also be represented in this form by using an appropriate dynamic model of the inputs, as shown in Section IV.

In the sequel, we refer to the intra-period time-steps as \textit{timeslots}. Specifically, the system is in the $k$-th timeslot if $t$ satisfies
$
mod (t,N) = k.
$
To simplify notation, define $A_k$ and $C_k$ to be $A(t)$ and $C(t)$ during the $k$-th timeslot, i.e.,
$$
A_k \triangleq A(k), \, C_k \triangleq C(k), \,\,\quad k=0,\ldots,N-1.
$$
Clearly, $A(t) = A_{mod(t,N)}$ and $C(t) = C_{mod(t,N)}$. 

The output, $y(t)$, is required to satisfy the constraint
\begin{equation}\label{eq:constraints}
y(t) \in \mathbb{Y}_k, \quad k = mod(t,N)
\end{equation}
for all $t \geq t_0$, with $\mathbb{Y}_k$ being a collection of constraint sets.

\subsection{Stability}
In general, stability of periodic systems cannot be inferred from individual $A_k$'s; rather, stability must be analyzed by studying the system behavior over one period. Specifically, let $\Phi(t,\tau)$ denote the state transition matrix \cite{Chen2012} of the system at time $t \in \mathbb{Z}^+$, starting from time $\tau \in \mathbb{Z}^+$, $\tau<t$. The state transition matrix over one period, i.e., $\Phi(\tau+N, \tau)$, is referred to as the {\it monodromy} matrix. The eigenvalues of the monodromy matrix, $\lambda_i$, are referred to as the {\it characteristic multipliers} and do not depend on the starting time $\tau$. The periodic system is asymptotically stable iff the characteristic multipliers satisfy $|\lambda_i| < 1$. In the sequel,  we denote the monodromy matrix starting at timeslot 0 by $\Phi$:
\begin{equation}\label{eq:monodromy}
\Phi = A_{N-1} \cdots A_1 A_0.
\end{equation}

\subsection{System lifting}
A convenient approach for analysis and design of periodic systems is to convert them to a time-invariant representation. This is accomplished by {\it lifting} the system into a space with a higher input/output dimensionality (for details, see \cite{Bittanti2000}). Specifically, define the lifted state, $\xi$, as the system state sampled once every period and the lifted output, $\widetilde{y}$, as the vectorized system outputs over one period. Since there are $N$ timeslots, there are $N$ distinct lifted systems. For example, sampling the states at timeslot 0, we obtain:\par
{\small
\begin{equation}\label{eq:liftedy}
\begin{aligned}
\xi(t) &\triangleq x(tN),\\
\widetilde{y}(t)&\triangleq [y(tN)^\top, y(tN+1)^\top, \ldots, \, y(tN + N-1)^\top]^\top.
\end{aligned}
\end{equation}}
The dynamics of the lifted system are described by
\begin{equation}\label{eq:lifted}
\begin{aligned}
\xi(t+1) &= \Phi \xi(t) \\
\widetilde{y}(t) &= C \xi(t)
\end{aligned}
\end{equation}
where $\Phi$ is the monodromy matrix defined in \eqref{eq:monodromy} and
\begin{equation}\label{eq:C}
C = \left[ \begin{array}{c} C_0 \\ C_1A_0 \\ \vdots \\ C_{N-1}A_{N-2}\cdots A_0 \end{array} \right].
\end{equation}
Clearly, the lifted system is a time-invariant system, whose output is exactly the outputs of the original system over one period stacked in a vector. Since the system matrix of \eqref{eq:lifted} is  the monodromy matrix, it follows that the lifted system is stable iff the original system is stable. While computations are usually more efficient on the original periodic system, the lifted system is useful for proving theoretical results.
%

\subsection{Assumptions}
The standing assumptions of this paper are listed below. Note that these assumptions are extensions of the standard assumptions in the literature of reference governors for LTI systems (for more details, please see \cite{Freuer2010,Gilbert1991a, Gilbert1999}).

\textbf{(A1)} We assume that the system matrices are in the form:
\begin{equation}\label{eq:formofAk}
A_k = \left[ \begin{array}{cc} A_{k,s} & A_{k,c} \\ \textbf{0} & I_d \end{array} \right], \,\, C_k = [C_{k,s}\,\,\,\, C_{k,c}],
\end{equation}
where $d\in \mathbb{Z}^+$ indicates the number of eigenvalues at $\lambda = 1$ with equal algebraic and geometric multiplicities. This form arises in predictive methods such as reference governors, wherein the inputs are held constant over the prediction horizon (see  Section IV). It is easy to show that, given the form of $A_k$ and $C_k$ above,  lifted system \eqref{eq:lifted} is also triangular and has the form
\begin{equation}\label{eq:phiform}
\Phi = \left[ \begin{array}{cc}\Phi_s & \Phi_c \\ \textbf{0} & I_d \end{array} \right], \,\, C = [C_s\,\,\,\, C_c],
\end{equation}
where the partitioning of $\Phi$ and $C$ is dimensionally consistent. We further assume that:

\textbf{(A2)} All eigenvalues of $\Phi_s$ are inside the open unit disk, which implies that the periodic system is Lyapunov stable. Note that unstable systems are not considered because reference governors are applied to closed loop systems that have been pre-stabilized with a tracking controller (see Fig. \ref{fig:RGblock});

\textbf{(A3)} The pair $(C, \Phi)$ is observable;

\textbf{(A4)} The constraint sets $\mathbb{Y}_0, \ldots, \mathbb{Y}_{N-1}$ are compact polytopes defined by:
\begin{equation}\label{eq:constraintPoly}
\mathbb{Y}_k = \{y: S_k y \leq \textbf{1} \},
\end{equation}
where $S_k \in \mathbb{R}^{q_k \times p}$, $q_k \in \mathbb{Z}^+$. Note that, without loss of generality, the right hand side of \eqref{eq:constraintPoly} is taken to be \textbf{1}, which satisfies $0 \in \mathrm{int}(\mathbb{Y}_k)$. For the lifted system \eqref{eq:lifted}, the output constraint set is the Cartesian product of the individual constraint sets:
$
\widetilde{y} \in \mathbb{Y} \triangleq \mathbb{Y}_{0} \times \cdots \times \mathbb{Y}_{N-1},
$
which, given \eqref{eq:constraintPoly}, becomes:
\begin{equation}\label{eq:S}
\mathbb{Y} = \{\widetilde{y}: S \widetilde{y} \leq \textbf{1} \},
\end{equation}
where $S$ is the block diagonal matrix with $S_0,\ldots,S_{N-1}$ on the diagonal. 


\section{Periodic Invariant Sets}

The concept of invariant sets is central to many predictive control design techniques, including reference governors (see, e.g., \cite{Gilbert1991a, Gilbert1999}). For LTI systems, the set of all initial conditions that lead to constraint satisfaction for all time is an invariant set referred to as the maximal output admissible set (MAS) \cite{Gilbert1991a}. In periodic systems, the situation is more complicated, because the set depends on the timeslot from which the dynamics start. As a result, periodic systems possess $N$ invariant sets that are cyclically related, as discussed below. In \cite{Freuer2010, Nguyen2014a}, these $N$ sets have been characterized for asymptotically stable periodic systems subject to state constraints. However, the case of Lyapunov stable systems, which is central to reference governors, and systems with \textit{output} constraints were not investigated. Furthermore, the geometric and algebraic relationship between these sets remain unexplored. This section addresses these issues.

\begin{definition}
The sets $\Omega_k,$ $k=0, \ldots, N-1$, are called the {\it maximal output-admissible sets} for system \eqref{eq:sys} if, starting from any initial condition in $\Omega_k$, and any initial time in the $k$-th timeslot, the output satisfies the constraints for all future times, that is, $\forall x(t_0) \in \Omega_k$ with $t_0$ satisfying $mod(t_0,N)=k$, we have that  $y(t) \in \mathbb{Y}_{mod(t,N)}, \forall t \geq t_0$.
\end{definition}
From Definition 1, it follows that $\Omega_k$ are cyclically related:
{\small
\begin{equation}\label{eq:cyclically}
\begin{aligned}
&x \in \Omega_k \Leftrightarrow C_k x \in \mathbb{Y}_k \,\,\mathrm{and}\,\, A_k x \in \Omega_{k+1}, \,\,\, 0\leq k\leq N-2\\
&x \in \Omega_{N-1} \Leftrightarrow C_{N-1} x \in \mathbb{Y}_{N-1} \,\,\mathrm{and}\,\, A_{N-1} x \in \Omega_{0}.
\end{aligned}
\end{equation}
}%
The above also implies that $\Omega_k$'s are periodic invariant sets for system \eqref{eq:sys}, because, starting from $\Omega_k$, the system state returns to $\Omega_k$ after one period.

To determine $\Omega_k$, define the sequence of sets $\{\Omega_k^j\}_{j=k}^{\infty}$:
$$
\Omega_k^j = \{x_0: x(k)=x_0, y(t) \in \mathbb{Y}_{mod(t,N)}, k \leq t \leq j\}.
$$
Since $\Omega_k^j$ are bounded and decreasing, the limit set exists; $\Omega_k$ is precisely the limit set: 
\begin{equation}\label{eq:omegak}
\Omega_k = \bigcap_{j \geq k} \Omega_k^j. 
\end{equation}
The computation of $\Omega_k$ above is based on expressing the output $y(t)$ explicitly as a function of the initial state, $x(k) = x_0$. From \eqref{eq:sys}, we have that $y(t) = C_{mod(t,N)} \left(\prod_{i=k}^{t-1} A_{mod(i,N)}\right) x_0$, $t \geq k$, where  $\prod$ denotes the left matrix multiplication. Thus, $y(t) \in \mathbb{Y}_{mod(t,N)}$ can be expressed as $S_{mod(t,N)} C_{mod(t,N)} \left(\prod_{i=k}^{t-1} A_{mod(i,N)}\right) x_0 \leq \textbf{1}$.  To make the above arguments clearer, we specialize the above to the case of $\Omega_0$: \par
\begin{equation}\label{eq:O1}{\small
\Omega_0 = \left\{x_0: \left[ \begin{array}{c} S_{0}C_{0} \\ S_{1}C_{1} A_0  \\ \vdots \\ S_{0}C_{0}A_{N-1} \cdots A_1 A_{0}  \\ S_{1}C_{1}A_0 A_{N-1} \cdots A_1 A_{0} \\ \vdots \end{array} \right] x_0 \leq  \textbf{1} \right\}.}
\end{equation}
Similar expressions can be found for $\Omega_1,\ldots, \Omega_{N-1}$. 

An important question arises regarding the computation of $\Omega_k$. In order to characterize and store $\Omega_k$ on a computer, it is important for the left hand matrix of \eqref{eq:O1} to be finite-dimensional. As it turns out, this is indeed the case, as shown in the theorems that follow.

\begin{definition}
The set $\Omega_k$ is said to be {\it finitely determined} if there exists a finite $j^*$ such that $\forall j \geq j^*$, $\Omega_k^{j} = \Omega_k^{j^*}$. We call $j^*$ the {\it admissibility index}.
\end{definition}

From this definition, it follows that if such $j^*$ exists, then $\Omega_k = \Omega_k^{j^*}$.

\begin{theorem}\label{thm:1}
Suppose system \eqref{eq:sys} is asymptotically stable (i.e., $d=0$ in \eqref{eq:phiform}) and assumptions (A1)-(A4) hold. Then, $\Omega_k$, $k=0,\ldots,{N-1}$, are finitely determined. In addition, $\Omega_k$ are compact sets and satisfy $0 \in \mathrm{int}(\Omega_k)$.
\end{theorem}
\begin{proof}
First, assume that the dynamics start from timeslot 0. With this assumption,  lifted system \eqref{eq:lifted} and the original system, \eqref{eq:sys}, describe the same outputs; therefore, MAS of \eqref{eq:lifted} is exactly $\Omega_0$. So, we prove the theorem for the MAS of \eqref{eq:lifted}. To do so, note that \eqref{eq:lifted} is asymptotically stable because \eqref{eq:sys} is (by the assumption in the theorem). We now invoke the results in \cite{Gilbert1991a}, which state that MAS for an asymptotically stable LTI system is always finitely determined. In addition, \cite{Gilbert1991a} shows that if the system is Lyapunov stable and observable, and the constraint set is compact and contains the origin in its interior, then MAS is also compact and contains the origin in its inerior. This completes the proof for $\Omega_0$. To prove the resutlts for $\Omega_k$, $k=1,\ldots,N-1$, we lift the system by sampling the state at the $k$-th timeslot and apply similar arguments.
\end{proof}


If the periodic system is not asymptotically stable (i.e., $d\geq 1$ in \eqref{eq:phiform}), a finitely determined inner approximation of $\Omega_k$ can be obtained by tightening the steady state constraint. This is a generalization of the results for LTI systems reported in \cite{Gilbert1991a}. The novelty here is that, since the system is periodic, steady state must be considered for the lifted system \eqref{eq:lifted}, \eqref{eq:phiform}. 
To simplify notation, define the matrix $\Gamma$:
\begin{equation}\label{eq:gamma}
\Gamma \triangleq [\textbf{0} \,\quad C_s(I-\Phi_s)^{-1}\Phi_c + C_c].
\end{equation}
The steady state output of the lifted system is given by $
\widetilde{y}_{ss} \triangleq \lim_{t\rightarrow \infty} \widetilde{y}(t) = \Gamma x_0.$ Define the set $\Omega_{ss}^\epsilon$ that corresponds to the tightened constraint $\widetilde{y}_{ss} \in (1-\epsilon) \mathbb{Y}$, for some $\epsilon \in (0, 1)$:
\begin{equation}\label{eq:tightened}
\Omega_{ss}^\epsilon = \{x_0: S\Gamma x_0 \leq (1-\epsilon)\textbf{1} \}.
\end{equation}
Then, $\Omega_k^\epsilon$ defined below is an inner approximation of $\Omega_k$ and is finitely determined:
\begin{equation}\label{eq:tightened2}
\Omega_k^\epsilon =  \Omega_{ss}^\epsilon  \cap \Omega_k.
\end{equation}
To prove this fact, we need the following lemma:
\begin{lemma}
The set $\Omega_{ss}^\epsilon$ defined in \eqref{eq:tightened} is independent of the timeslot at which the state is sampled in the lifting process.
\end{lemma}
\begin{proof}
Let $\widetilde{y}$ and $\widetilde{z}$ be the lifted outputs with the state sampled at timeslot 0 and timeslot $k$, $1\leq k\leq N-1$:
$$
\widetilde{y} = [y(tN)^\top, \ldots, \, y(tN + N-1)^\top]^\top 
$$
$$
\widetilde{z} = [y(tN+k)^\top, \ldots, \, y(tN + k + N-1)^\top]^\top. 
$$
Lyapunov stability implies that $\widetilde{y}(t+1)=\widetilde{y}(t)$ as $t \rightarrow \infty$. This observation, together with the definition of $\widetilde{y}$, implies that $y(t)$ is $N$-periodic. Therefore, $\widetilde{z}(t)$ must be a permutation of $\widetilde{y}(t)$. This suggests that lifting the system by sampling the state at a different timeslot leads to a reshuffling of the inequalities in \eqref{eq:tightened}, which does not change the set. 
\end{proof}

\begin{theorem}
For all $\epsilon \in (0,1)$, the sets $\Omega_0^\epsilon, \ldots, \Omega^{\epsilon}_{N-1}$,  are finitely determined.
\end{theorem}


\begin{proof}
We first prove that $\Omega_0^\epsilon$ is finitely determined, using arguments involving the lifted system \eqref{eq:lifted}. Similar to the proof of Theorem \ref{thm:1}, assume that the dynamics start from  timeslot 0. Then, MAS for \eqref{eq:lifted} is exactly $\Omega_0$. In \cite{Gilbert1991a}, it is shown that the MAS for a Lyapunov stable LTI system of the form \eqref{eq:lifted}, \eqref{eq:phiform} is finitely determined if the steady state constraint is tightened. Therefore, the MAS for \eqref{eq:lifted}, \eqref{eq:phiform} with tightened steady state \eqref{eq:tightened} is finitely determined. Hence, $\Omega_0^\epsilon$ is finitely determined. 

To prove finite determinism of $\Omega_k^\epsilon$, $k=1,\ldots,N-1$, we lift the system by sampling the state at the $k$-th timeslot.  According to Lemma 1, $\Omega_{ss}$ also defines a tightened steady state constraint for this lifted system. Therefore, the arguments in the above paragraph hold without modification. This completes the proof.
\end{proof}

\begin{figure*}
    \centering
    \begin{subfigure}[t]{0.3\textwidth}
        \centering
        \includegraphics[width=\textwidth]{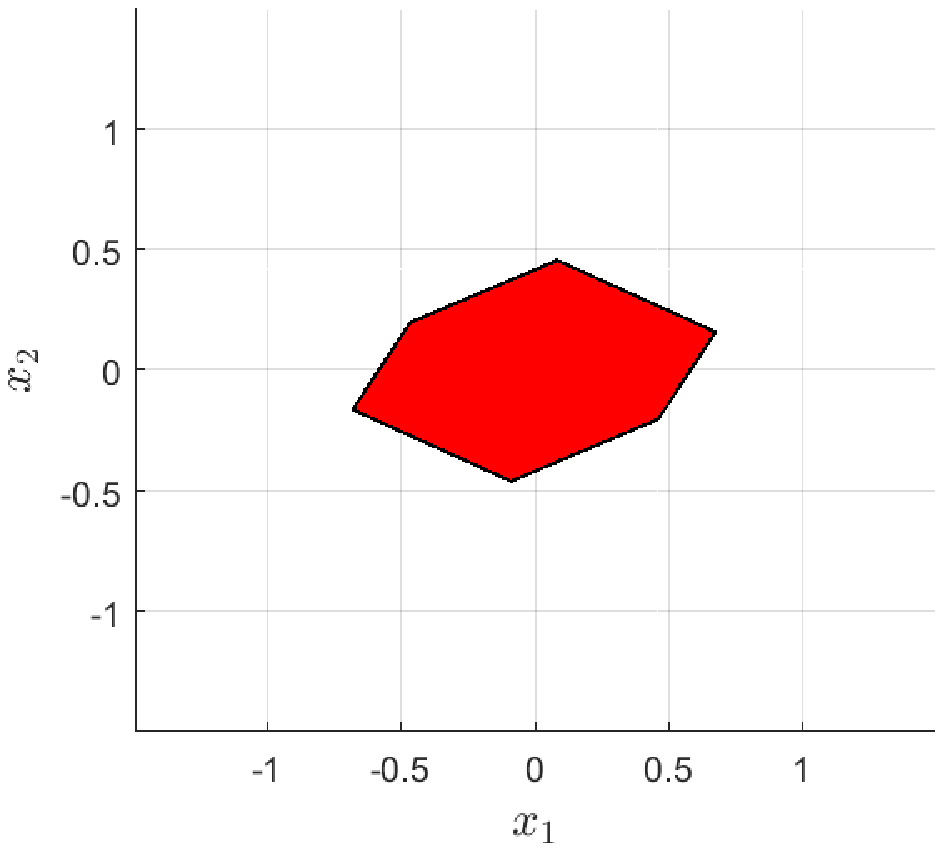}
        \caption{$\Omega_0$}
    \end{subfigure}%
    ~ 
    \begin{subfigure}[t]{0.3\textwidth}
        \centering
        \includegraphics[width=\textwidth]{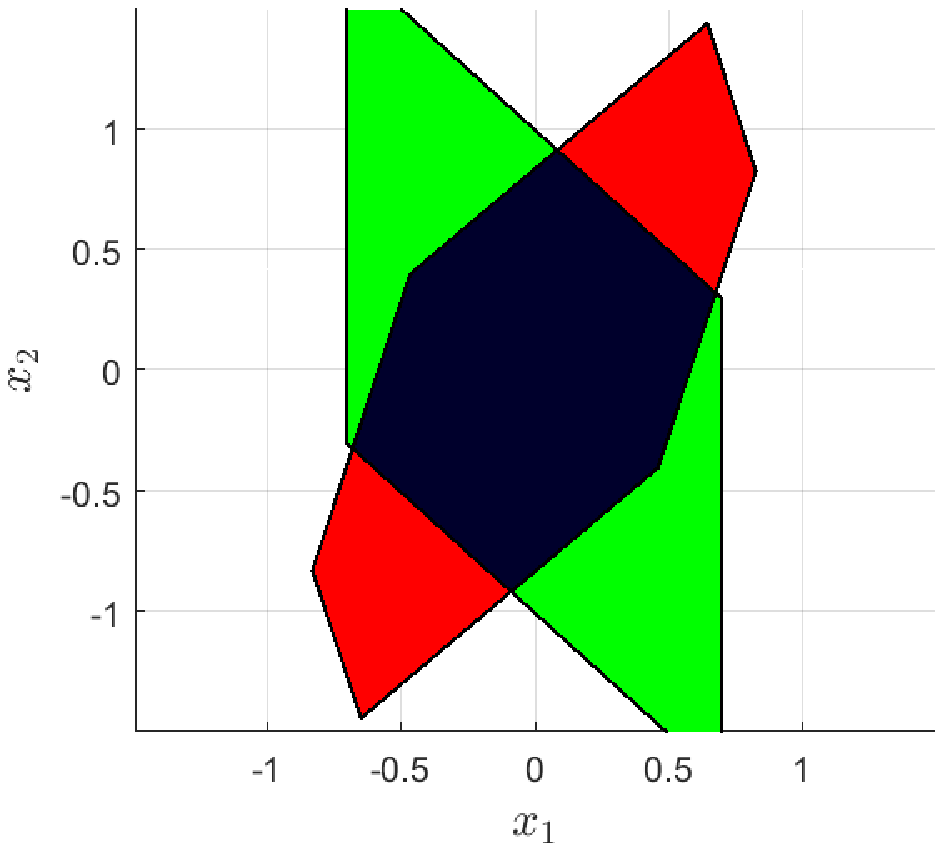}
        \caption{$\Omega_1$}
    \end{subfigure}%
    ~ 
    \begin{subfigure}[t]{0.3\textwidth}
        \centering
        \includegraphics[width=\textwidth]{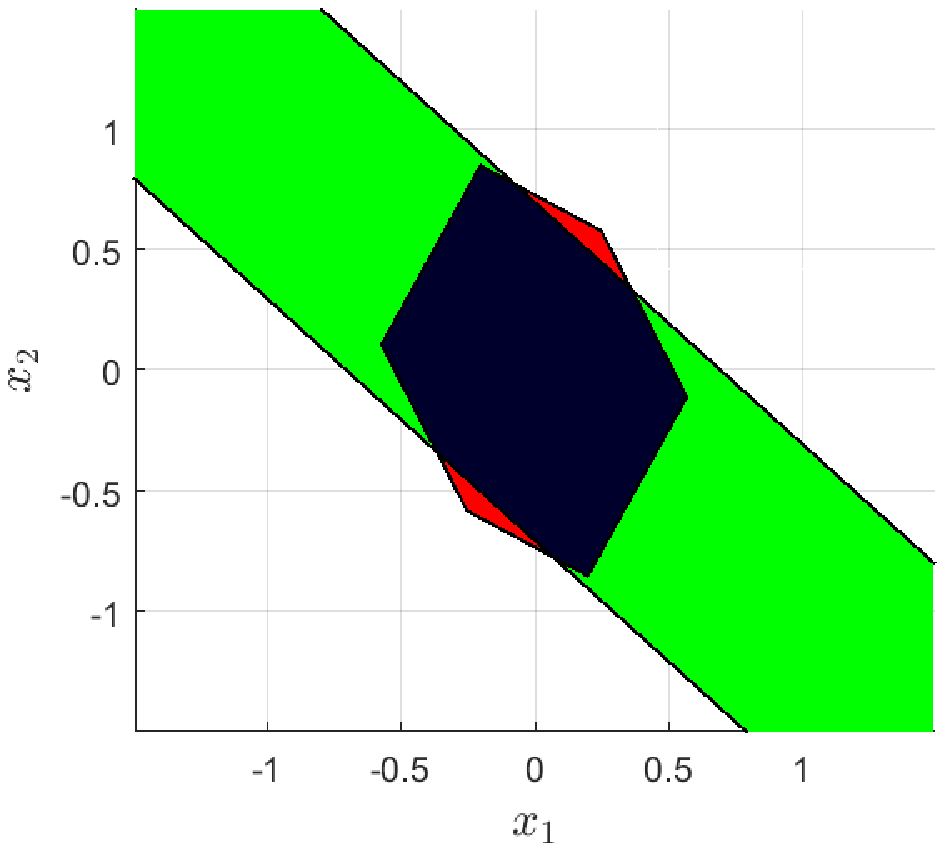}
        \caption{$\Omega_2$}
    \end{subfigure}
    \caption{$\Omega_k$'s for the example in Section III. The left plot shows $\Omega_0$. In the center plot, red shows $\Omega_0$ scaled and rotated via $(A_2A_1)^{-1}$, green shows the new half-spaces introduced by $S_1 C_1x \leq \textbf{1}$ and $S_2 C_2 A_1x \leq \textbf{1}$, black is $\Omega_1$, which is the intersection of red and green. In the right plot, red shows $\Omega_0$ scaled and rotated via matrix $A_2^{-1}$, green shows the new half-spaces introduced by $S_2 C_2x \leq \textbf{1}$, black is $\Omega_2$, which is the intersection of red and green.}
            \label{fig:ex1}
\end{figure*}
It can easily be shown that, similar to $\Omega_k$, $\Omega_k^\epsilon$ is  compact and satisfies $0 \in \mathrm{int}(\Omega_k^\epsilon)$.

According to the above theorems and the form of the inequalities in \eqref{eq:O1}, for asymptotically stable periodic systems (i.e., $d=0$ in \eqref{eq:phiform}), the $\Omega_k$'s can be represented by polytopes of the form 
$$\Omega_k = \{x: H_k x \leq \textbf{1}\},$$ where $H_k$ is a finite-dimensional matrix. If the periodic system is Lyapunov stable with $d\geq 1$, the $\Omega_k$'s may not be finitely determined. In this situation, we use $\Omega_k^\epsilon$'s instead. These sets can be represented by polytopes of the form $$\Omega_k^\epsilon = \{x: H_k x \leq {h}_k\},$$ where $H_k$ and $h_k$ are finite-dimensional matrices. Note from \eqref{eq:tightened}, \eqref{eq:tightened2} that all rows of $h_k$ are equal to 1 except the rows that correspond to the tightened steady state \eqref{eq:tightened}, which are equal to $1-\epsilon$.

\begin{remark} Suppose system \eqref{eq:sys} is Lypaunov stable and $\Omega_k$'s are not finitely determined. To characterize $\Omega_k^\epsilon$, $\epsilon$ cannot be chosen too small as the number of inequalities required to describe $\Omega_k^\epsilon$ approaches infinity as $\epsilon$ tends to zero. On the other hand, $\epsilon$ cannot be chosen too large, because $\Omega_k^\epsilon$ would become a conservative inner approximation of $\Omega_k$, which could lead to loss of performance in control applications. Therefore, a trade-off must be made.
\end{remark}

Now that we have established finite determinism of $\Omega_k$'s, we are ready to investigate their relationship. As we will show, the novel outcome of this investigation is that only one of these sets needs to be computed and stored for real-time applications. We begin with the following theorem, which describes the nested relationship between $\Omega_k$'s.
\begin{theorem}\label{thm:3}
Assume system \eqref{eq:sys} is asymptotically stable (i.e., $d=0$ in \eqref{eq:phiform}) and let $\Omega_k$, $k=1,\ldots, N-1$, be described by $\Omega_k = \{x: H_k x \leq \textbf{1}\}$ for a finite-dimensional matrix $H_k$. Then,  $\Omega_{k-1}$, can be described by $\Omega_{k-1}=\{x: H_{k-1} x \leq \textbf{1}\}$, where
$$
H_{k-1} = \left[ \begin{array}{c} S_{k-1}C_{k-1} \\ H_k (A_{k-1}) \end{array} \right].
$$
For $k=0$, replace $k-1$ by $N-1$ in the above.
\end{theorem}
\begin{proof}
Let $k$ be any index between 1 and $N-1$. By \eqref{eq:cyclically}, we have that $x \in \Omega_{k-1}$ iff $C_{k-1} x \in \mathbb{Y}_{k-1}$ and $A_{k-1} x \in \Omega_k$. Note that $C_{k-1} x \in \mathbb{Y}_{k-1}$ can be expressed as $S_{k-1}C_{k-1}x \leq \textbf{1}$ and $A_{k-1} x \in \Omega_k$ can be expressed as $H_k A_{k-1} x \leq \textbf{1}$. Since both of these inequalities must hold, $\Omega_{k-1}$ can be characterized as stated in the theorem. The case of $k=0$ is proven in the same way except $k-1$ must be replaced by $N-1$.
\end{proof}

Next, we prove the Lyapunov-stable case. For this, we need a lemma:
\begin{lemma}\label{lemma:mat}
For all $\epsilon \in (0,1)$, the sets $\Omega_k^\epsilon$ are related by \par
{\small
\begin{equation}
\begin{aligned}
&x \in \Omega_k^\epsilon \Leftrightarrow C_k x \in \mathbb{Y}_k \,\,\mathrm{and}\,\, A_k x \in \Omega_{k+1}^\epsilon, \,\,\, 0\leq k\leq N-2\\
&x \in \Omega_{N-1}^\epsilon \Leftrightarrow C_{N-1} x \in \mathbb{Y}_{N-1} \,\,\mathrm{and}\,\, A_{N-1} x \in \Omega_{0}^\epsilon.
\end{aligned}
\end{equation}
}
\end{lemma}
\begin{proof}
Due to the position of \textbf{0}'s and $I_d$ in $\Gamma$ and $A_k$, we have that $\Gamma A_k = \Gamma$, $\forall k$. Let $k$ be any index from 0 to $N-2$. From \eqref{eq:tightened}, \eqref{eq:tightened2}, we have that $x \in \Omega_k^\epsilon$ iff $S\Gamma x \leq (1-\epsilon)\textbf{1}$, $S_kC_kx\leq \textbf{1}$, $S_{k+1}C_{k+1}A_k x\leq \textbf{1}$, etc. By the first part of the proof, we can write $S\Gamma x \leq (1-\epsilon)\textbf{1}$ as $S\Gamma A_k x \leq (1-\epsilon)\textbf{1}$. Therefore, $x \in \Omega_k^\epsilon$ iff $S_kC_kx\leq \textbf{1}$ (i.e., $C_k x \in \mathbb{Y}_k$) and $A_k x \in \Omega_{k+1}^\epsilon$. The case of $k=N-1$ can be proven similarly.
\end{proof}

\begin{theorem}
Assume system \eqref{eq:sys} is Lyapunov stable with $d\geq 1$ and let $\Omega_k^\epsilon$ in \eqref{eq:tightened2} be described by $\Omega_k^\epsilon = \{x: H_k x \leq h_k\}$ for finite-dimensional matrices $H_k$ and $h_k$. Then,  $\Omega_{k-1}^\epsilon$ can be described by $\Omega_{k-1}^\epsilon=\{x: H_{k-1} x \leq h_{k-1}\}$, where
$$
H_{k-1} = \left[ \begin{array}{c} S_{k-1}C_{k-1} \\ H_k (A_{k-1}) \end{array} \right], h_{k-1} = \left[ \begin{array}{c} \textbf{1} \\ h_{k} \end{array} \right],
$$
For $k=0$, replace $k-1$ by $N-1$ in the above.
\end{theorem}
\begin{proof}
Similar to the proof of Theorem \ref{thm:3}, let $k$ be any index between 1 and $N-1$. By Lemma \ref{lemma:mat}, we have that $x \in \Omega_{k-1}^\epsilon$ iff $C_{k-1} x \in \mathbb{Y}_{k-1}$ and $A_{k-1} x \in \Omega_k^\epsilon$. Note that $C_{k-1} x \in \mathbb{Y}_{k-1}$ can be expressed as $S_{k-1}C_{k-1} x\leq \textbf{1}$ and $A_{k-1} x \in \Omega_k^\epsilon$ can be expressed as $H_k A_{k-1} x \leq h_k$. Since both of these inequalities must hold, $\Omega_{k-1}^\epsilon$ can be characterized as stated in the theorem. The case of $k=0$ is proven in the same way except $k-1$ must be replaced by $N-1$ in the above.
%
%
\end{proof}

According to the above theorems, $H_{k-1}$ and $h_{k-1}$ can be expressed explicitly as a function of $H_k$ and $h_k$. The following corollary uses the above ideas to relate $H_{k}$'s and $h_{k}$'s to $H_0$ and $h_0$. 

\begin{corollary}
Let $\Omega_0$ (or $\Omega_0^\epsilon$ if the system is not asymptotically stable) be described by $\{x: H_0 x \leq h_0\}$ for finite-dimensional matrices $H_0$ and $h_0$. Then, $\Omega_k$ (or $\Omega_k^\epsilon$), $k=1,\ldots, N-1$, can be described by $\{x: H_k x \leq h_k \}$, where \par
{\small $$
H_{N-1} = \left[ \begin{array}{c} S_{N-1}C_{N-1} \\ H_0 (A_{N-1}) \end{array} \right], h_{N-1} =  \left[ \begin{array}{c} \textbf{1} \\ h_0 \end{array} \right], 
$$
$$
H_{N-2} = \left[ \begin{array}{c} S_{N-2}C_{N-2} \\ S_{N-1}C_{N-1}A_{N-2} \\ H_0 (A_{N-1}A_{N-2}) \end{array} \right], h_{N-2} =  \left[ \begin{array}{c} \textbf{1} \\ \textbf{1} \\ h_0 \end{array} \right], 
$$
$$
 \ldots\,\,\, H_{1} = \left[ \begin{array}{c} S_{1}C_{1} \\ S_{2}C_{2}A_{1} \\ \vdots \\ S_{N-1} C_{N-1} A_{N-2}\cdots A_{1} \\ H_0 (A_{N-1}\cdots A_2 A_1) \end{array} \right], h_{1} =  \left[ \begin{array}{c} \textbf{1} \\ \textbf{1} \\ \vdots \\ \textbf{1} \\  h_0 \end{array} \right].
$$}
\end{corollary}

This corollary has three important implications for $\Omega_k$'s (these implications are also valid for $\Omega_k^\epsilon$'s). First, if $\Omega_k$ is finitely determined for some $k$, then all other $\Omega_i$, $i\neq k$, are also finitely determined. In addition, the admissibility index of $\Omega_i$ is at most $j^*+N-1$, where $j^*$ is the admissibility index of $\Omega_k$ and $N$ is the period.

Second, not all $H_k$'s and $h_k$'s need to be stored in memory for real-time control. Instead, only $H_0$ and $h_0$ must be stored, together with the matrices that arise in the corollary. For example, to store $H_{N-1}$, in addition to $H_0$ and $h_0$, we need to store $A_{N-1}$ and the result of $S_{N-1}C_{N-1}$. Similarly, to store $H_{N-2}$, we need to store the results of $S_{N-2}C_{N-2}$,  $S_{N-1}C_{N-1} A_{N-2}$, and $A_{N-1}A_{N-2}$. This may significantly reduce the memory burden in an embedded setting, at the expense of an increase in processing burden (due to real-time matrix multiplication and addition). We analyze this trade-off in Section IV.

Third, it can be seen that the sets $\Omega_1,\ldots,\Omega_{N-1}$ are geometrically related to $\Omega_0$ via simple transformations: to obtain $\Omega_k$, the set $\Omega_0$ is first scaled and rotated. Then, new half spaces are introduced and, finally, the new half spaces are intersected with the scaled and rotated $\Omega_0$. Note that the scaling and rotation are determined by the inverse of the product of $A_k$'s that multiply $H_0$ in the last rows of $H_{k}$'s in Corollary 1. If the inverse does not exist, the resulting transformation will be unbounded in one or more directions. These ideas are illustrated in Fig. \ref{fig:ex1} for the following example\footnote{Plots are generated by MPT 3.0 Toolbox in Matlab \cite{M.HercegM.KvasnicaC.Jones2013}.}: $N=3$ and
\begin{equation}\label{eq:example}
\begin{aligned}
& A_0 = \left[ \begin{array}{cc} 1 & 0 \\ 0 & 2 \end{array} \right],  A_1 = \left[ \begin{array}{cc} 0.8 & -0.5	 \\ 0.2 & 0.5 \end{array} \right], S_{0,1} = \left[ \begin{array}{c} 1  \\ -1 \end{array} \right], \\
& A_2 = \left[ \begin{array}{cc} 0 & -0.8 \\ 0.8 & 0 \end{array} \right], S_2=  \left[ \begin{array}{c} \frac{10}{7} \\ -\frac{10}{7}\end{array} \right], C_{0,1,2}=[1,1]^\top.
\end{aligned}
\end{equation}

Finally, below we present an algorithm for computing $\Omega_0$ (or $\Omega_0^\epsilon$ for Lyapunov-stable systems). The sets $\Omega_1, \ldots, \Omega_{N-1}$ can be computed using Corollary 1.

\noindent \textbf{Algorithm 1.}

\textbf{Inputs:} $A_k$, $C_k$, $S_k$, $\epsilon$.

\textbf{Outputs:} $H_0$ and $h_0$ such that $\Omega_0 = \{x:H_0 x\leq h_0\}$.

\begin{enumerate}
\item Initialization: Set $P = I$, $t=0$, $c=0$. If $d \geq 1$ in \eqref{eq:phiform}, compute $\Phi$, $C$, $S$, $\Gamma$ from \eqref{eq:monodromy}, \eqref{eq:C}, \eqref{eq:S}, \eqref{eq:gamma} and set 
$$
H_0 = S \Gamma, \, h_0 = (1-\epsilon) \textbf{1},
$$
otherwise, set $H_0 = [ \,]$ and $h_0 =   [\, ]$. 

\item Set $Y = S_{mod(t,N)} C_{mod(t,N)} P$.
\item For each row of $Y$, $y_i$, solve the linear program
$$
M = \max_{z} y_i z \quad\,\,\, \mathrm{subject\, to}\,\,\, H_0z\leq h_0.
$$
If $M > 1$ or the problem is unbounded, $y_i$ is not redundant, so augment $y_i$ and ${1}$ to $H_0$ and $h_0$: 
$$
H_0 = \left[ \begin{array}{c} H_0 \\ y_i \end{array}\right], \, h_0 = \left[ \begin{array}{c} h_0 \\ {1} \end{array}\right].
$$
\item If any $y_i$ is not redundant, set $c=0$.
\item If all $y_i$ are redundant, set $c = c + 1$. If $c=N$, terminate the algorithm. 
\item Set $P = A_{mod(t,N)} P$, $t = t + 1$. Go to step 2.
\end{enumerate}

\section{Reference Governors for Linear Periodic Systems}

In this section, we use the results of Section III to introduce two reference governor (RG) formulations for linear periodic systems. 

Consider the system shown in Fig. \ref{fig:RGblock}, whose closed loop plant is a single-input multi-output $N$-periodic system:
\begin{equation}\label{eq:sysIn}
\begin{aligned}
{x}(t+1) &= \bar{A}_k {x}(t) + \bar{B}_k v(t)\\
y(t) &=\bar{C}_k {x}(t) + \bar{D}_k v(t)
\end{aligned}
\end{equation}
where $k = mod(t,N)$. It is assumed that the closed loop plant contains a stabilizing tracking controller. Hence, the above periodic system is asymptotically stable. The output, $y$, is required to satisfy the constraint $y(t) \in \mathbb{Y}_{mod(t,N)}$, similar to \eqref{eq:constraintPoly}. The role of RG is to select $v(t)$, at every $t$, to ensure that: (i) $y(t) \in \mathbb{Y}_{mod(t,N)}$ for all $t$, and (ii) $v(t)$ is as close to $r(t)$ as possible to minimize loss of tracking performance. Note that the second equation in \eqref{eq:sysIn} assumes that $y(t)$ depends on the selected $v(t)$, which is only possible in a sampled data system if the selected $v(t)$ is applied to the system without any delays or computational overhead. We allow this for simplicity.

In the LTI case \cite{Gilbert1991a}, RG uses the MAS of an augmented system defined by the original dynamics augmented with constant input dynamics $v(t+1)=v(t)$. In real time, RG computes $v(t)$ to be $r(t)$ passed through a low pass filter with variable time constant. Specifically, $v(t) = v(t-1) + \kappa (r(t)-v(t-1))$, where $\kappa \in [0,1]$ is the time constant and is computed by solving the following linear program: $\max_{\kappa \in [0,1]} \kappa$ subject to $(x(t),v(t))$ belonging to the MAS of the augmented system, where $x$ is the measured (or observed) state and $v(t)$ is given by the above filter equation. If $v(t)=r(t)$ is a feasible input, then $\kappa = 1$ and $v(t)=r(t)$. If significant constraint violation is predicted, then $\kappa = 0$ and $v(t)=v(t-1)$. Note that if $v(-1)$ is such that $(x(0), v(-1))$ belongs to MAS, then $v(t)=v(t-1)$  guarantees constraint satisfaction for all times (as it is included in the augmented dynamics), implying recursive feasibility of RG.

For periodic systems, RG can be developed in a similar manner. As we show, the main difference with the LTI case is that, at every timestep, RG must employ the MAS at the current timeslot, i.e., $\Omega_{mod(t,N)}^\epsilon$. Furthermore, the periodicity of the system allows for richer formulations of RG. Below, we present two formulations, discuss their stability and recursive feasibility, present algorithms for their computation, and highlight their differences. 

\subsection{Formulation 1 -- Fixed Input}\label{sec:1}

The first formulation, similar to RG for LTI systems, assumes constant input dynamics: $v(t+1)=v(t)$. Defining ${x}_{aug} = [x^\top, v]^\top$, the augmented system has the form \eqref{eq:sys}, \eqref{eq:formofAk}, where
\begin{equation}\label{eq:ACexample}
A_k = \left[ \begin{array}{cc} \bar{A}_k & \bar{B}_k \\ \textbf{0} & 1 \end{array} \right], \,\, C_k = [\bar{C}_k\,\,\,\, \bar{D}_k].
\end{equation}
The above augmented system is Lyapunov stable. Thus, we compute $\Omega_k^\epsilon$ using Algorithm 1 and Corollary 1, for some $0 <\epsilon \ll 1$. Now, the reference governor can be implemented in real time by either updating $v(t)$ at the beginning of the period, i.e., at $t=iN$, $i\in \mathbb{Z}^+$, or updating $v(t)$ at every $t$. The former implementation is  simpler and only requires $\Omega_0^\epsilon$, but it has the downside that changes to $v(t)$ will only occur at timeslot 0, potentially causing a delay in the system that could be as large as $N-1$ samples. Therefore, we only explore the latter implementation: at every time $t$,  RG computes $v(t)$ as
$$
v(t) = v(t-1) + \kappa (r(t)-v(t-1)),
$$
where $\kappa$ is computed by solving the linear program
\begin{equation}\label{eq:RG1}
\begin{aligned}
&\underset{\kappa}{\text{maximize}} & & \kappa\\
& \text{subject to} & & \kappa \in [0,1] \\ &&& (x_0,v_0) \in \Omega_{mod(t,N)}^\epsilon \\ &&& v_0 = v(t-1) + \kappa (r(t)-v(t-1)) \\ &&& x_0 = {x}(t)
\end{aligned}
\end{equation}
Note that ${x}(t)$ (measured or estimated), $v(t-1)$, and $r(t)$ are known parameters in the above linear program.

{\it Feasibility and Stability of Formulation 1:} Since $\Omega_k^\epsilon$'s are computed with $v(t+1)=v(t)$ as part of the augmented dynamics, it follows that $\kappa = 0$ (i.e., $v(t)=v(t-1)$ or equivalently $v(t+1)=v(t)$) is always feasible, implying that this RG formulation is recursively feasible. Moreover, since $v(t)$ is a low pass filtered version of $r(t)$, the overall system with RG is stable.

{\it Numerical solution:} We now present an efficient algorithm for solving \eqref{eq:RG1}. Note that the constraints in \eqref{eq:RG1} can be equivalently written as $({x}(t), v(t-1) + \kappa (r(t)-v(t-1))) \in \Omega_{mod(t,N)}^\epsilon$. To simplify notation, let $\tau = mod (t,N)$ denote the current timeslot. Representing $\Omega_\tau^\epsilon$ by $\Omega_\tau^\epsilon = \{x:H_\tau x \leq h_\tau\}$, and partitioning $H_\tau$  as $H_\tau = [H_{\tau x},\,\, H_{\tau v}]$, the constraints in \eqref{eq:RG1} become:
\begin{equation}\label{eq:RGsimple}
\kappa (r(t)-v(t-1)) H_{\tau v} \leq h_\tau - H_{\tau x}{x}(t) - H_{\tau v}v(t-1).
\end{equation}
The $i$-th row of the above is in the form $\kappa a_i \leq b_i$. Note that, because of recursive feasibility, we have that $(x(t),v(t-1)) \in \Omega_\tau$, i.e., $H_{\tau x}{x}(t) + H_{\tau v}v(t-1) \leq h_\tau$. This implies that $b_i \geq 0, \forall i$. Therefore, \eqref{eq:RG1} can be solved by iteratively looping through all the  inequalities in \eqref{eq:RGsimple} and finding the maximum $\kappa \in [0,1]$ that satisfies all the inequalities. This leads to the following algorithm for solving \eqref{eq:RG1}:

\noindent \textbf{Algorithm 2.}

\textbf{Inputs:} $H_{\tau x}$, $H_{\tau v}$, $h_\tau$, $r(t)$, $v(t-1)$, $x(t)$.

\textbf{Output:} $\kappa$.

\begin{enumerate}
\item Set $i = 1$. Set $\kappa = 1$.
\item Compute the $i$-th row of \eqref{eq:RGsimple}. Let the left and right hand sides of the inequality be denoted by $a_i$ and $b_i$.
\item Compute $\gamma = \frac{b_i}{a_i}$. 
\item If $0 \leq \gamma < \kappa$, set $\kappa = \gamma$.  
\item Set $i=i+1$. If $i$ is larger than the number of rows of $H_{\tau x}$, terminate. Otherwise, go to step 2.
\end{enumerate}

Note that the above algorithm terminates in finite, known time. In fact, the number of iterations is equal to the number of rows of $H_\tau$, which is known {\it a priori}. With this simple algorithm, there is no need for an advanced solver such as interior point or simplex. 

{\it Space-time trade-offs:} According to Corollary 1, instead of storing all $H_k$'s and $h_k$'s in memory, it suffices to only store $H_0$, $h_0$, and few other matrices that arise in the corollary. This reduces the memory burden of the RG, at the expense of an increase in processing requirements. We now make these claims concrete. The notations and assumptions in this analysis are as follows. We let $N$, as before, denote the system period and let $n$ and $p$ denote the dimensions of $x$ and $y$ in \eqref{eq:sysIn} respectively. We assume that $H_0$ is computed using Algorithm 1 and $H_1, \ldots, H_{N-1}$ are computed using Corollary 1 without any additional inspection of redundancy. We denote the number of rows of $H_0$ by $m$, which, together with \eqref{eq:ACexample}, implies that $H_0$ is an $m \times (n+1)$ matrix. We further assume that the control loop runs on a single processor without parallelization. For simplicity, we assume that $S_k$ is $q\times p$, i.e., $q_k=q, \forall k$ in \eqref{eq:constraintPoly}. Finally, we refer to the approach of storing all $H_k$'s and $h_k$'s in memory as ``complete storage'' and that of storing only $H_0$ and $h_0$ and the matrices that arise in Corollary 1 as ``partial storage''.

Recall from Section III that all rows of $h_k$'s are equal to 1 except the rows that correspond to the tightened steady state \eqref{eq:tightened}, which are equal to $1-\epsilon$. Therefore, $h_k$'s do not need to be stored in memory. Instead, it suffices to store the value of $\epsilon$ and the row numbers in which $1-\epsilon$ appears. For this reason, we do not consider $h_k$'s in this analysis. 

Using the matrices in Corollary 1, under complete storage, a total of $Nm(n+1) + \frac{N (N-1)}{2}q(n+1)$ floating point numbers are required to represent all $H_k$'s. Under partial storage, however, it can be shown that $m(n+1) + \frac{N(N-1)}{2}q(n+1)+(N-1)(n)(n+1)$ floating points are needed. Subtracting the latter expression from the former yields the memory savings due to partial storage: 
$$
\text{Floating point numbers saved} = (N-1)(n+1)(m-n).
$$
In most processors, floating points are represented by 32 bits (4 byte). Therefore,
\begin{equation}\label{eq:memorysavings}
\text{Total memory saved} = 4(N-1)(n+1)(m-n) \text{ bytes}.
\end{equation}
Clearly, if $m > n$, memory savings are positive. Therefore, partial storage should only be used for systems in which $m > n$ (which is typically the case). Furthermore, the gains are linear in $N$ and $m$, implying that partial storage yields significant gains for large $N$ and $m$.

The above memory savings come at the expense of additional processing burden. To investigate, we begin with an example scenario: assume the current timeslot is $\tau=N-1$. Using Corollary 1 and the structure of $A_k$ in \eqref{eq:ACexample}, we can express $H_{\tau x}$, $H_{\tau v}$, $h_\tau$ in \eqref{eq:RGsimple} in terms of $H_{0x}$, $H_{0v}$, $h_0$ as:\par
{\small $$
H_{\tau x} = \left[ \begin{array}{c} S_{\tau }\bar{C}_{\tau } \\ H_{0x} \bar{A}_{\tau } \end{array} \right], H_{\tau v} = \left[ \begin{array}{c} S_{\tau } \bar{D}_{\tau }  \\ H_{0x}\bar{B}_{\tau }+H_{0v} \end{array} \right], h_{\tau } = \left[ \begin{array}{c} \textbf{1} \\ h_0 \end{array} \right].
$$}%
The first row of the above must be stored regardless of partial or complete storage. Therefore, the disadvantage of partial storage is that the matrices $H_{0x}\bar{A}_\tau $ and $H_{0x}\bar{B}_{\tau }+H_{0v}$ must be computed in real time. However, note that $H_{\tau x}$ is used in \eqref{eq:RGsimple} as a product $H_{\tau x}{x}$. This implies that  $H_{0x}\bar{A}_\tau$ appears in \eqref{eq:RGsimple} as the product $H_{0x}\bar{A}_\tau {x}$. Therefore, instead of computing $H_{0x}\bar{A}_\tau $, the product $\bar{A}_\tau  {x}$ can first be computed  (this requires $n^2$ multiplications and $n(n-1)$ additions) and the result can be used in the calculations in \eqref{eq:RGsimple}, which requires the same number of arithmetic operations regardless of partial or complete storage. The computation of $H_{0x}\bar{B}_{\tau }+H_{0v}$ requires $mn$ multiplications and $mn$ additions. It can be shown that the numbers found above are the same for all timeslots $\tau=0,\ldots,N-1$. Therefore, as compared with complete storage, partial storage increases the total number of multiplications and additions by
\begin{equation}\label{eq:computationalincrease}
\text{Additional arithmetic operations: } n(2m+2n-1)
\end{equation}
in each timestep.

\begin{figure}
     \centering
     \includegraphics[scale=0.55]{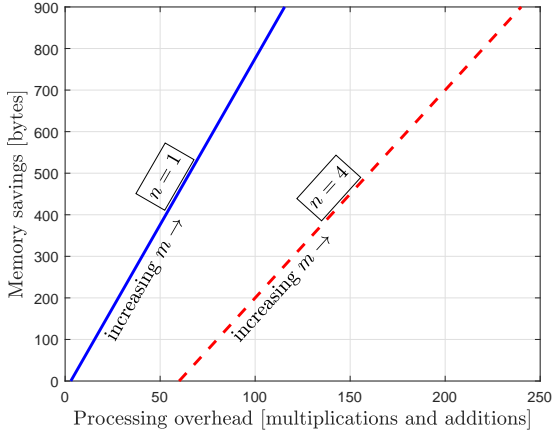}
     \caption{Space-time trade-off for the partial storage approach. The above lines are generated for $N=3$. For the solid blue curve, $n=1$ and $m$ is varied from 1 to 58. For the dashed red curve, $n=4$ and $m$ is varied from 4 to 28. The plot shows that memory savings come at the expense of an increase in processing overhead. Moreover, the relationship is linear in $m$ and has a higher slope for smaller $n$'s. This indicates that the greatest benefits of  partial storage are realized for large $m$ and small $n$.}
     \label{fig:computations}
\end{figure}

To illustrate the above, Fig. \ref{fig:computations} shows the memory savings, \eqref{eq:memorysavings}, versus the additional processing burden, \eqref{eq:computationalincrease}, due to the partial storage approach. The plots are generated for $N=3$, two $n$'s, and several values of $m$. Clearly, the savings in memory are associated with an increase in processing burden. Furthermore, as seen by the slope of the lines, smaller $n$'s lead to the greatest savings. Using this figure and equations \eqref{eq:memorysavings},  \eqref{eq:computationalincrease}, an engineering decision can be made as to whether or not the partial storage approach would be beneficial for the application at hand.

{\it Numerical Example:} The performance of this RG formulation is illustrated on a numerical example obtained using Matlab. The matrices $\bar{A}_k$, $\bar{C}_k$ are the same as  ${A}_k$, ${C}_k$ in \eqref{eq:example}, ${S}_k$ are the same as those in \eqref{eq:example}, and $N=3$, $B_0=(-2,1), B_1 = (1,0), B_2 = (8,-1)$, $D_{0,1,2} = 0$, $\epsilon=0.05$. For the purpose of illustration, it is assumed that all states are measured. The response for a pulse input $r(t)$ is shown in Fig. \ref{fig:ex2_1}. As can be seen, the RG successfully shapes the reference signal to prevent constraint violation. Note that, for this simple example, $m = 22$, so Algorithm 2 converges in 22 iterations. Furthermore, under complete storage, 846 byte of memory are required to store $\Omega_k^\epsilon$'s, whereas under partial storage, only 384 byte are required. This is a 54\% improvement, which comes at the cost of an increase in the number of real-time multiplications and additions by 94.

\begin{figure*}
    \centering
    \begin{subfigure}[t]{0.44\textwidth}
        \centering
        \includegraphics[width=\textwidth]{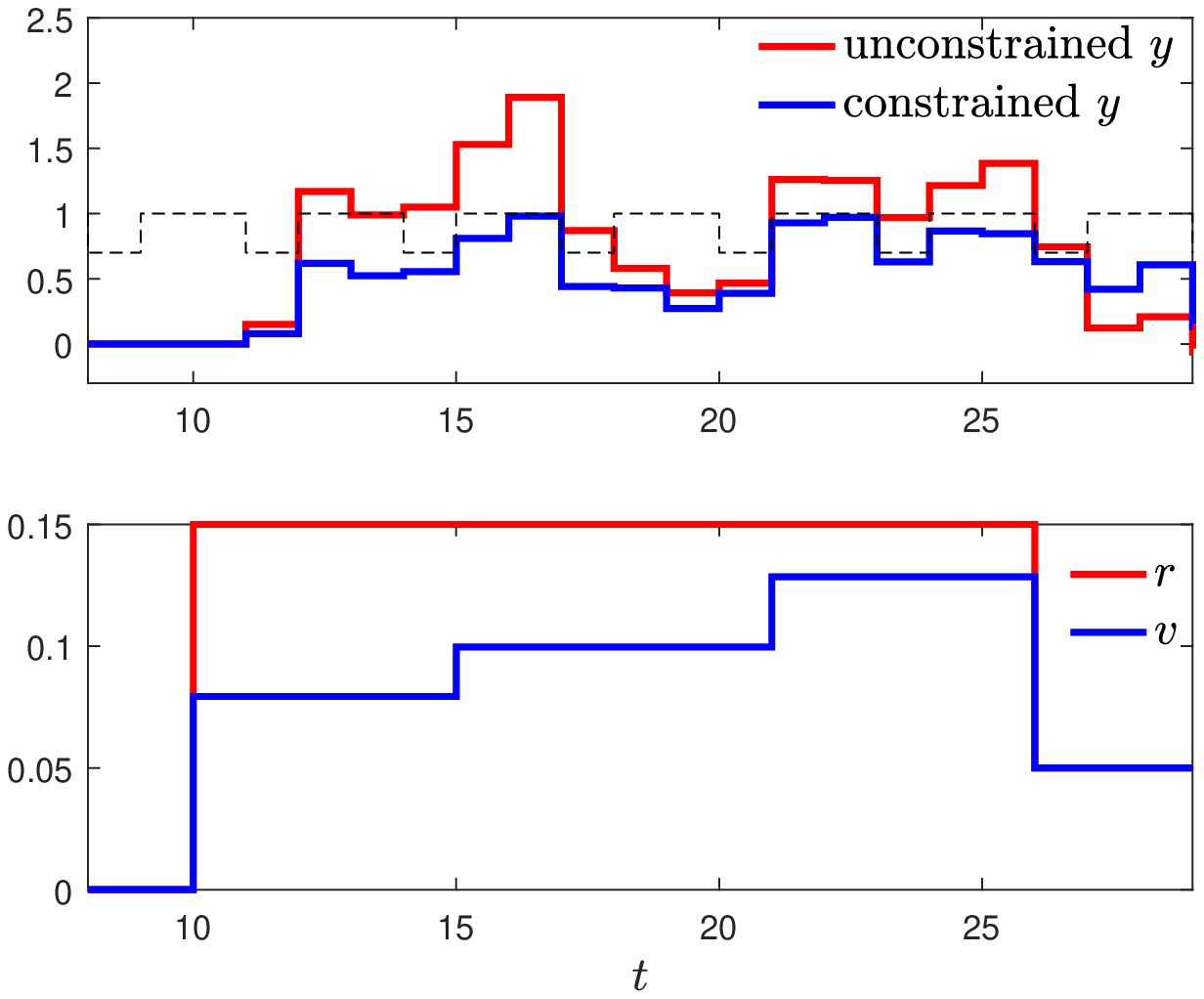}
        \caption{Formulation 1 (Fixed input)}
        \label{fig:ex2_1}
    \end{subfigure}%
    \quad\quad\quad
    \begin{subfigure}[t]{0.44\textwidth}
        \centering
        \includegraphics[width=\textwidth]{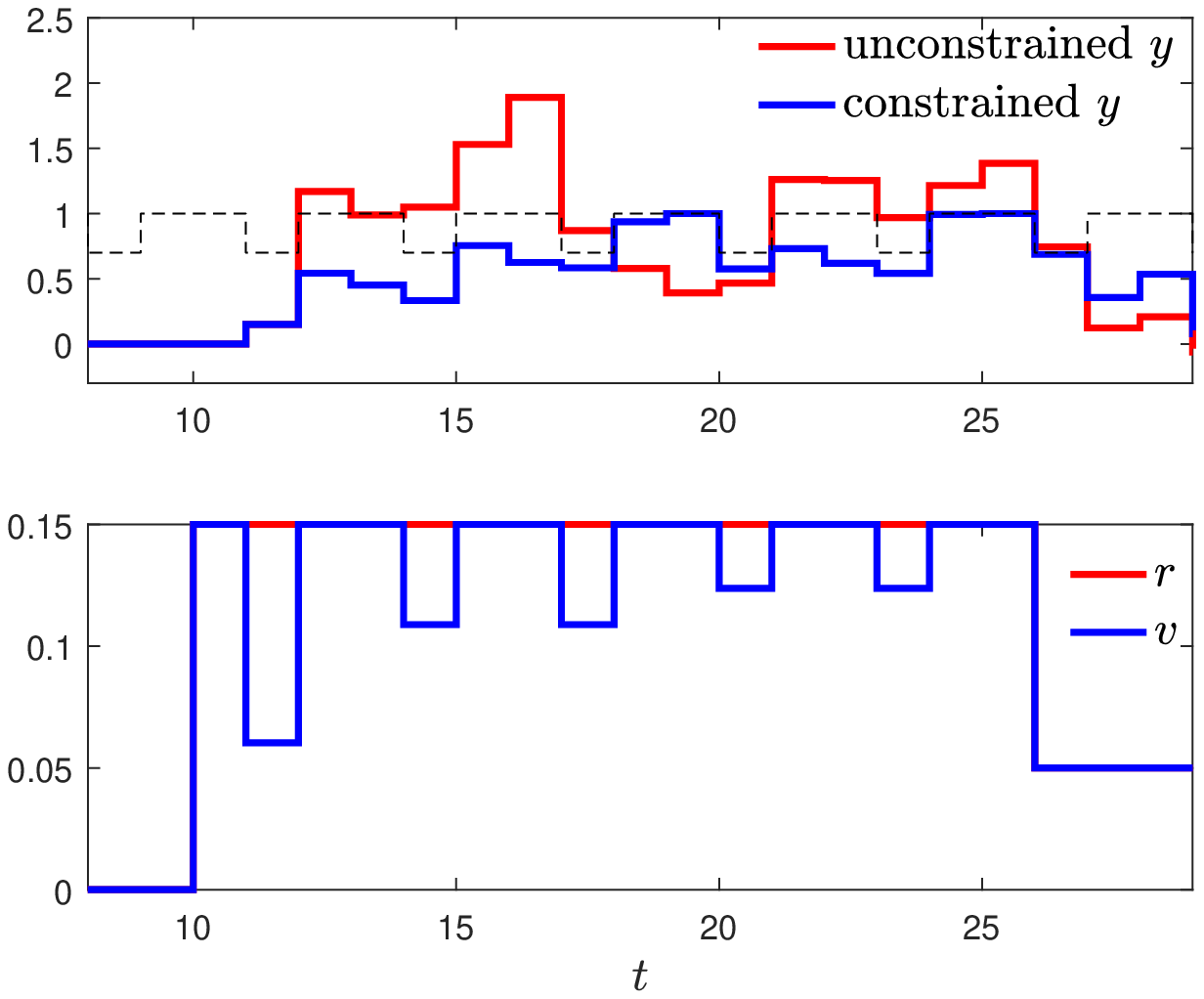}
        \caption{Formulation 2 (Periodic input)}
        \label{fig:ex2_2}
    \end{subfigure}%
 
    \caption{RG example in Section IV. The left and right plots show the response of the first and second formulations, respectively. In both plots, the applied reference signal is a step from 0 to 0.15 at $t=10$ followed by a step down to 0.05 at $t=26$. The dashed black lines represent the constraint. As can be seen, the unconstrained (i.e., ungoverned) output (red curve) violates the constraints. However, the constrained output (i.e., governed output, shown by blue curve) satisfies the constraint in both cases. Note that formulation 1 exploits periodicity only in its use of $\Omega_k^\epsilon$. In formulation 2, however, periodicity is also exploited in $v(t)$. As a result of this additional complexity, $v(t)$ more closely follows $r(t)$, which is desired.}
            \label{fig:ex2}
\end{figure*}

\subsection{Formulation 2 -- Periodic Input}
A second, richer formulation of RG can be developed by taking full advantage of the periodic nature of the system. Specifically, define $v_0(t), \ldots, v_{N-1}(t)$ to be  $N$ new states and let $v(t)=v_{mod(t,N)}(t)$ (i.e., $v(t)$ is periodically assigned to $v_k$'s). To compute the MAS, we assume $v_k$'s are dictated by constant dynamics:
$$
v_k(t+1) = v_k(t),\,\,\, k=0,\ldots,N-1.
$$
Defining $x_{aug} = [{x}^\top, v_0, \ldots, v_{N-1}]^\top$, the augmented system has the form \eqref{eq:sys}, where
$$
A_k = \left[ \begin{array}{cc} \bar{A}_k & \bar{B}_k e_{k+1}^\top \\ \textbf{0} & I_N \end{array} \right], \,\, C_k = [\bar{C}_k\,\,\,\, \bar{D}_ke_{k+1}^\top].
$$
The vectors $e_1,\ldots,e_{N}$ are the standard basis vectors in $\mathbb{R}^N$. The sets $\Omega_k^\epsilon$ can be computed for this system using Algorithm 1 and Corollary 1, for some $0 < \epsilon \ll 1$.

Let the current timeslot be denoted by $\tau=mod(t,N)$. In this formulation, at each $t$, RG updates only $v_\tau$, leaving the remaining $v_k$'s unchanged. Specifically, at each timestep, RG solves the following linear program in the variable $\kappa$: \par
{\small
\begin{equation*}
\begin{aligned}
&\underset{\kappa}{\text{maximize}} & & \kappa \\
& \text{subject to} & & \kappa \in [0,1],  \\ &&& (x_0,v_0,\ldots,v_{N-1}) \in \Omega_{\tau}^\epsilon,\,\,\, \tau = mod(t,N) \\ &&& v_\tau = v_\tau(t-1) + \kappa(r(t)-v_\tau(t-1))\\ &&& v_i = v_i(t-1), i \neq \tau \\
&&&x_0 = {x}(t)\\
\end{aligned}
\end{equation*}}%
Note that $x(t)$, $v_k(t-1)$, $k=0,\ldots,N-1$, and $r(t)$ are known parameters in this optimization program. Once the optimal $\kappa$ is found, RG updates $v_\tau$ and applies the updated value to $v(t)$ as follows:
$$
v_\tau(t) = v_\tau(t-1)+\kappa (r(t)-v_\tau(t-1)), \,\,\,\, v(t) = v_\tau(t).
$$
To solve the above optimization problem, the constraints can be simplified and an algorithm similar to Algorithm 2 can be developed. Furthermore, stability and recursively feasibility can be proven for this formulation.  Since the approach is similar to that used in the first formulation, we do not present these results for the sake of brevity. 

For this formulation,  the partial storage approach also leads to  memory savings at the expense of additional computing needs:
$$
\text{Total memory savings} = 4(N-1)(n+3)(m-n) \text{ bytes}
$$
$$
\text{Additional arithmetic operations: } n(6m+2n-1)
$$
Clearly, the trends are similar to those in Formulation 1.

The performance of this RG is illustrated in Fig. \ref{fig:ex2_2} for the same $r(t)$ used in the example in Section \ref{sec:1} (shown in Fig. \ref{fig:ex2_1}). In contrast to formulation 1, this formulation takes advantage of the periodicity of the system to more effectively shape $v(t)$ to prevent constraint violation. Specifically, the governed reference is allowed to vary and repeat periodically within the period, leading to $v(t)$'s that are closer to $r(t)$. Note that, for this example, $m$ is computed to be 24, so an extension of Algorithm 2 for this formulation would converge in 24 iterations. Furthermore, under complete storage, 1560 byte of memory are required to store $\Omega_k^\epsilon$'s, whereas under partial storage, only 680 byte are required. This is a an improvement of 880 bytes (56\% improvement), which comes at the cost of an increase in the number of real-time multiplications and additions by 294.

\begin{remark}
A comparison between formulations 1 and 2 is in order. Note that formulation 1 results in $v(t)$'s that converge to a constant value, whereas formulation 2 results in $v(t)$'s that may periodically repeat at steady state. Therefore, from a practical standpoint, formulation 2 is advantageous in situations where different setpoints are desired in different timeslots. Examples are multi-rate systems and cyclic control systems, where a distinct plant is controlled in each timeslot. From an implementation perspective, formulation 1 augments the dynamics with 1 additional state whereas formulation 2 introduces $N$ new states. This implies that the memory and computational overhead of formulation 2 are much higher. For instance, in the example presented in Fig. \ref{fig:ex2}, under complete storage, formulation 1 requires 846 bytes of memory to store $\Omega_k^\epsilon$'s, whereas formulation 2 requires 1560 bytes, an increase by  84\%.
\end{remark}

\begin{remark}
The ideas presented in this paper can be extended to multi-input periodic systems, wherein $v(t)$ and $r(t)$ are vectors of dimension greater than 1. However, care must be taken with the formulation of RG in these cases. Following formulation 2, one can let $v_\tau (t) = v_\tau (t-1) + \kappa (r(t)-v_\tau (t-1))$, where $\kappa$ is now a diagonal matrix with $\kappa_0, \ldots, \kappa_{g-1}$ on the diagonals, with $g$ being the dimension of $v$. The values of $\kappa_i$ are obtained by solving an optimization problem. One may be tempted to maximize $\sum \kappa_i$ as a proper extension of the single variable case. However, this formulation is not robust as it may lead to non-unique solutions. Therefore, we suggest using a quadratic program instead. For example, one could minimize $\sum \|v_i-r_i\|^2$  subject to $\kappa_i \in [0,1]$ and the state and input belonging to the periodic MAS of the augmented system. Note that this reasoning is similar to the case of vector reference governors \cite{Kolmanovsky2014}, where, rather than a linear program, a quadratic program is solved to optimize $\kappa$ over multiple channels. 
\end{remark}

\section{Conclusions}

This paper considers reference governors and the maximal output admissible sets (MAS)  for linear periodic systems. We extend the earlier results in the literature to the case of Lyapunov stable periodic systems with output constraints, showing that an inner approximation of MAS for these systems is finitely determined. Furthermore, we study the relationship between these sets and show that they are related via simple transformations. The significance of this result is that, instead of all $N$ sets, only one needs to be computed and stored in memory for the purpose of control. This reduces the memory requirements of a reference governor scheme, at the expense of additional computing requirements. We present a thorough analysis of this trade-off. In the second part, we present two RG formulations, discuss their feasibility criteria and implementation, and illustrate their efficacy using numerical simulations. Complete definitions, proofs, and algorithms are included, with the hope that both practitioners and researchers can find the results of this paper accessible.

Future work includes: robust RG for periodic systems under unknown disturbances, robust RG for periodic systems under polytopic uncertainty, RG with output feedback, and extension of the work to command governors. A practical application will appear in a future publication.





\section*{Acknowledgments}

The author greatly appreciates the valuable discussions with Dr. Mads Almassalkhi, Dr. Uro\v{s} Kalabi\'{c}, and Dr. Ilya Kolmanovsky. 

\bibliographystyle{IEEEtran}
\bibliography{IEEEabrv,periodicRG_TAC}

\begin{IEEEbiography}
    [{\includegraphics[width=1in]{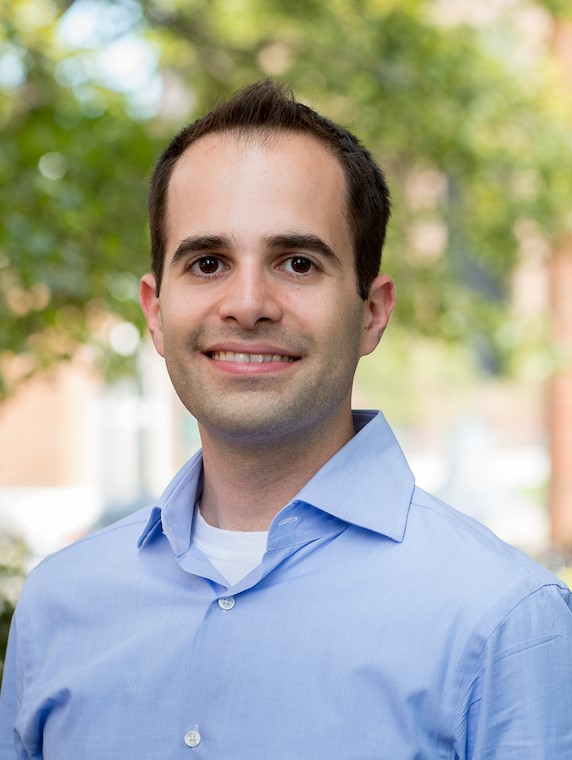}}]{Hamid R. Ossareh} obtained his BASc (EE) from the University of Toronto in 2008, and MASc (EE), MS (Mathematics), and PhD (EE) degrees from the University of Michigan, Ann Arbor in 2010, 2012, and 2013, respectively. From 2013-2016, he was with Ford Research and Advanced Engineering as a research engineer, where he investigated advanced control of automotive systems. Since 2016, he has been an assistant professor in the Department of Electrical and Biomedical Engineering at the University of Vermont. His current research interests include constraint management using reference and command governors, control of nonlinear stochastic systems, and control of periodic systems. He has published 16 articles in peer-reviewed journals and conferences and holds 11 US patents.
    
Dr. Ossareh was a recipient of the Natural Sciences and Engineering Research Council of Canada Post-Graduate Scholarship, the Towner Prize for Best Graduate Instructor from the University of Michigan, and the Chief Engineer's award from Ford Motor Company.
\end{IEEEbiography}

\end{document}